\documentclass[10pt]{amsart}
\usepackage{amsfonts}
\usepackage{ifthen}
\usepackage{amsthm}
\usepackage{amsmath,mathrsfs}
\usepackage{graphicx}
\usepackage{amscd,amssymb,amsthm}
\usepackage{color}
\usepackage{hyperref}

\newcounter{minutes}
\divide\time by 60
\newcounter{hours}
\multiply\time by 60 \addtocounter{minutes}{-\time}

\newtheorem{theorem}{Theorem}
\newtheorem{lemma}{Lemma}

\newtheorem{corollary}{Corollary}

\title[Volterra functions]{Monotonicity Properties and functional inequalities for the Volterra and incomplete Volterra functions }

\author[K. Mehrez, S. M. Sitnik]{Khaled Mehrez\;and Sergei M. Sitnik}
\address{Khaled Mehrez\newline
D\'epartement de Math\'ematiques, Universit\'e de Kairouan, Tunisia \textit{and}\newline
D\'epartement de Math\'ematiques, Facult\'ee des sciences de Tunis, Universit\'e Tunis El Manar, Tunisia.}
\email{k.mehrez@yahoo.fr}

\address{Sergei M. Sitnik\newline
Belgorod State National Research University (BSU), Belgorod, Russia.}
\email{Sitnik@bsu.edu.ru}



\begin{document}

\def\thefootnote{}
\footnotetext{ \texttt{File:~\jobname .tex,
          printed: \number\year-0\number\month-\number\day,
          \thehours.\ifnum\theminutes<10{0}\fi\theminutes}
} \makeatletter\def\thefootnote{\@arabic\c@footnote}\makeatother

\maketitle

\begin{abstract}
In this paper we prove  some monotonicity, log--convexity and log--concavity properties for the Volterra and incomplete Volterra functions.  Moreover, as
consequences of these results, we present some functional inequalities (like Tur\'an type inequalities) as well as we determined sharp upper and lower bounds for the normalized incomplete Volterra functions in terms of weighted power means.
\end{abstract}

\section{Introduction}
 The Volterra and  related  functions  to  be  considered  are  defined  in  the  following  way (see \cite{ER}--\cite{App2}):
\begin{eqnarray}
\nu(x)&=&\int_0^\infty\frac{x^t}{\Gamma(t+1)}dt,\\
\nu(x,\alpha)&=&\int_0^\infty\frac{x^{t+\alpha}}{\Gamma(t+\alpha+1)}dt,\\
\mu(x,\beta)&=&\int_0^\infty\frac{x^t t^\beta}{\Gamma(t+1)\Gamma(\beta+1)}dt\\
\mu_(x,\beta,\alpha)&=&\int_0^\infty\frac{x^{t+\alpha} t^\beta}{\Gamma(t+\alpha+1)\Gamma(\beta+1)}dt,
\end{eqnarray}
where $\alpha,\beta>-1$ and $x>0,$ but some particular notations are usually adopted in the special cases
\begin{equation}
\begin{split}
\alpha&=\beta=0,\;\;\nu(x)=\mu(x,0,0)\\
\alpha&\neq0,\beta=0,\;\;\nu(x,\alpha)=\mu(x,0,\alpha)\\
\alpha&=0,\beta\neq0,\;\;\mu(x,\beta)=\mu(x,\beta,0).
\end{split}
\end{equation}

Volterra functions were introduced by Vito Volterra in 1916. Its theory was thoroughly developed by Mhitar M. Dzhrbashyan, his and his coathors results were summed up in the monograph \cite{Dzh} in 1966. In this book many important results on Volterra functions, known and new, were gathered and introduced. They include results on: formulas for different integrals, derivatives and limits, representations via Mittag--Leffler and related functions, integral representations for different kind of domains, real and complex, solution of integral equations and inversion of integral transforms, asymptotics, connections with Laplace and Mellin transforms and Parseval identities. An important class of results is connected with generalized Fourier transforms on half--axes and a system of rays, including applications to integral representations of analytic functions on corner and other special complex domains. Many results on Volterra functions were also gathered in two books of A. Apelblat \cite{App1}--\cite{App2}, for important application cf. also \cite{GaMa}.

In this paper we define the incomplete Volterra functions $\mu(x,\beta,\alpha,s)$ and $\mu^*(x,\beta,\alpha,s)$ by
\begin{equation}
\mu(x,\beta,\alpha,s)=\int_s^\infty\frac{x^{t+\alpha} t^\beta}{\Gamma(t+\alpha+1)\Gamma(\beta+1)}dt,\;\;\textrm{and}\;\;\mu^*(x,\beta,\alpha,s)=\int_0^s\frac{x^{t+\alpha} t^\beta}{\Gamma(t+\alpha+1)\Gamma(\beta+1)}dt,
\end{equation}
where $\alpha,\beta>-1,x>0$ and $\geq0.$ Throughout this paper, we denote by $G_\beta$ and $G_\beta^*$  the normalized functions
\begin{equation}
G_\beta(x,\alpha,s)=\frac{\mu(x,\beta,\alpha,s)}{\mu(x,\beta,\alpha,0)}=\frac{\mu(x,\beta,\alpha,s)}{\mu(x,\beta,\alpha)},\;\;\textrm{and}\;\;G_\beta^{*} (x,\alpha,s)=\frac{\mu^*(x,\beta,\alpha,s)}{\mu(x,\beta,\alpha)}.
\end{equation}

Starting with this representations  we will obtain new properties of the Volterra and incomplete Volterra functions, including conditions for complete monotonicity, log-convexity and log-concavity in upper parameters. Moreover, as consequences of these results, we presented some functional inequalities, in particular some Tur\'an type inequalities are proved. So this paper is a continuation of author's results on Tur\'an type and related inequalities for a new class of Volterra functions, cf. \cite{SiMe1}--\cite{Me1}.

The plan of the paper is the following. In section 2,  we offer  complete monotonicity, log-convexity and log-concavity properties  in upper parameters of the Volterra function. As applications, we derive some inequalities for these special functions. Moreover, we prove that the following Tur\'an type inequality holds for  Volterra functions $\Gamma(\beta+1)\mu(x,\beta,\alpha),$ more precisely, for
all odd integers $n\geq1$ and real numbers $x > 0,\alpha,\beta>-1$ we have
$$\left(\frac{\partial^{n-1}\left[\Gamma(\beta+1)\mu(x,\beta,\alpha)\right]}{\partial \beta^{n-1}}\right)\left(\frac{\partial^{n+1}\left[\Gamma(\beta+1)\mu(x,\beta,\alpha)\right]}{\partial \beta^{n+1}}\right)-\left(\frac{\partial^{n}\left[\Gamma(\beta+1)\mu(x,\beta,\alpha)\right]}{\partial \beta^{n}}\right)^2\geq0.$$
In section 3, we proved that  functions $G_\beta$ and $G_\beta^*$ are log-concave on $(-1,\infty),$ As applications we derived two Tur\'an type inequalities involving these functions. Moreover, we derive that the function $G_\beta$ (resp. $G_\beta^*$) is increasing (resp. decreasing) in $\beta.$ As a consequence  we found sharp upper and lower bounds for  functions $G_\beta$ and $G_\beta^*$ in terms of weighted power means. Also  we prove that the function $G_\beta$ is  subadditive.

The next definitions will be used in the paper.

A function $f:[a,b]\subseteq\mathbb{R}\rightarrow (0,\infty)$  is said to be  log--convex if its natural logarithm $\log (g)$ is convex, that is, for all $x,y\in[a,b]$ and $\lambda\in[0,1]$ we have
$$f(\lambda x+(1-\lambda)y)\leq \left[f(x)\right]^\lambda\left[f(y)\right]^{1-\lambda}.$$
If the above inequality is reversed then $f$ is called a log--concave function. If $f$ is differentiable, then $f$ is log--convex (log--concave) if and only if $f^\prime/f$ is increasing (decreasing).

A function $g:[a,b]\subseteq(0,\infty)\rightarrow (0,\infty)$  is said to be geometrically (or multiplicatively) convex if $f$ is convex with respect to the geometric mean, i.e. if for all $x, y\in[a,b]$ and $\lambda\in[0,1]$ we have
	$$g(x^\lambda y^{1-\lambda})\leq[f(x)]^\lambda[f(y)]^{1-\lambda}.$$
	It is also known that if $f$ is differentiable, then $f$ is geometrically convex  if and only if $u\mapsto xf^\prime(x)/f(x)$ is increasing on $[a,b].$
	
	A function $h:(0,\infty)\longrightarrow\mathbb{R}$ possessing derivatives of all order is called a completely monotonic function if $$(-1)^n h^{(n)}(x)\geq0,\;\textrm{for\;all}\;x>0,\;n\in\mathbb{N}_0=\left\{0,1,2,...\right\}.$$
	The celebrated Bernstein Characterization Theorem gives a sufficient condition for the complete
monotonicity of a function $h$ in terms of the existence of some nonnegative locally integrable
function $H(x),\;x>0$ , referred to as the spectral function, for which
$$h(x)=\int_0^\infty e^{-xt}H(t)dt.$$
For some difficulties and wide--spread errors concerning generalization of this result to absolutely monotonic functions cf. \cite{Sit1}.

\section{Monotonicity Properties and functional inequalities for the Volterra function}

First, we prove complete monotonicity, log--convexity and log--concavity properties of the Volterra function $\mu(x,\beta,\alpha)$

\begin{theorem}\label{T1}Let $\alpha,\beta>-1.$ The following assertions are true:\\
\noindent \textbf{a.} The function $x\mapsto \mu(x,\beta,\alpha)$ is geometrically convex on $(0,\infty).$\\
\noindent \textbf{b.} The function $x\mapsto \mu(x^{-1},\beta,\alpha)$ is completely monotonic and log-convex on $(0,\infty)$ for all $\alpha\geq0.$\\
\noindent \textbf{c.} The function $\alpha\mapsto \mu(x,\beta,\alpha)$ is  log-concave on $(-1,\infty)$ for all $x>0.$\\
\noindent \textbf{d.} The function $\beta\mapsto \mu(x,\beta,\alpha)$ is log-concave on $(-1,\infty).$\\
\noindent \textbf{e.} The function $\beta\mapsto H(\beta)=\Gamma(\beta+1)\mu(x,\beta,\alpha)$ is log-convex on $(-1,\infty).$\\
\noindent \textbf{f.} The function $x\mapsto e^{-x}-\nu(-x)$ is completely monotonic on $(0,\infty).$\\
\end{theorem}

\begin{proof}\textbf{a.} By using the Rogers--H\"older--Riesz inequality, we get
\begin{equation}
\begin{split}
\mu(x^\lambda y^{1-\lambda}, \beta,\alpha)&=\int_0^\infty \frac{x^{\lambda (t+\alpha)}y^{(1-\lambda)(t+\alpha)}t^\beta}{\Gamma(t+\alpha+1)\Gamma(\beta+1)}dt\\
&=\int_0^\infty \left[\frac{x^{t+\alpha}t^\beta}{\Gamma(t+\alpha+1)\Gamma(\beta+1)}\right]^\lambda \left[\frac{y^{t+\alpha}t^\beta}{\Gamma(t+\alpha+1)\Gamma(\beta+1)}\right]^{1-\lambda}dt\\
&\leq\left[\int_0^\infty\frac{x^{t+\alpha}t^\beta}{\Gamma(t+\alpha+1)\Gamma(\beta+1)}dt\right]^\lambda\left[\int_0^\infty\frac{y^{t+\alpha}t^\beta}{\Gamma(t+\alpha+1)\Gamma(\beta+1)}dt\right]^{1-\lambda}\\
&=\left[\mu(x,\beta,\alpha)\right]^\lambda\left[\mu(y,\beta,\alpha)\right]^{1-\lambda}.
\end{split}
\end{equation}
Thus, the function $x\mapsto \mu(x,\beta,\alpha)$ is geometrically convex on $(0,\infty)$ for each $\alpha,\beta>-1.$

\textbf{b}. By using the fact that the function $x\mapsto x^{-(t+\alpha)}$ is completely monotonic function on $(0,\infty),$ we obtain
\begin{equation*}
\begin{split}
(-1)^n \frac{\partial^n \mu(x^{-1},\beta,\alpha)}{\partial x^n}&=\int_0^\infty \frac{t^{\beta}}{\Gamma(t+\alpha+1)\Gamma(\beta+1)}\left(\frac{(-1)^n\partial^n x^{-(t+\alpha)}}{\partial x^n}\right)dt\geq0
%&=\int_0^\infty (t+\alpha)^n\frac{\left((-1)^n\partial^n e^{-(t+\alpha)\log(x)}/\partial x^n\right)}{\Gamma(t+\alpha+1)}dt
\end{split}
\end{equation*}
for all $x\in(0,\infty)$ and $\alpha\geq0.$ Thus, the function $x\mapsto \mu(x^{-1},\beta,\alpha)$ is completely monotonic and consequently is log--convex, since every completely monotonic function is log--convex, see \cite[p. 167]{WI}.\\
%c. For $n\in\mathbb{N},$ we have
%$$(-1)^n \frac{\partial^n \nu(x,\alpha)}{\partial \alpha^n}=(-\log(x))^n\int_0^\infty \frac{x^{t+\alpha}}{\Gamma(t+\alpha+1)}\geq0,$$
%for all $x\in(0,1).$ So, the function $\alpha\mapsto \nu(x,\alpha)$ is completely monotonic and consequently is log-convex on $[0,\infty)$ and $x\in(0,1).$ Now, we proved the log-convexity property for the function $\alpha\mapsto \nu(x,\alpha)$ for all $x>0.$ Again, by using the H\"older-Rogers inequality, we obtain for all $\alpha_1,\alpha_2\geq0$ and $\lambda\in[0,1]$ we have
%\begin{equation}
%\begin{split}
% \nu(x,\alpha_1 \lambda+(1-\lambda)\alpha_2)&=\int_0^\infty \frac{x^t x^{\alpha_1\lambda}x^{\alpha_2(1-\lambda)}}{\Gamma(t+\alpha+1)}\\
%&=\int_0^\infty \left[\frac{x^t x^{\alpha_1}}{\Gamma(t+\alpha+1)}\right]^\lambda\left[\frac{x^t x^{\alpha_1}}{\Gamma(t+\alpha+1)}\right]^{1-\lambda}dt\\
%&\leq\left[\int_0^\infty\frac{x^t x^{\alpha_1}}{\Gamma(t+\alpha+1)}dt\right]^\lambda\left[\int_0^\infty\frac{x^t x^{\alpha_1}}{\Gamma(t+\alpha+1)}dt\right]^{1-\lambda}\\
%&=\left[\nu(x,\alpha_1)\right]^\lambda\left[\nu(x,\alpha_2)\right]^{1-\lambda}.
%\end{split}
%\end{equation}

\textbf{c.} We set
$$f_{x,\alpha}(t,\beta)=\frac{x^{t+\alpha}t^\beta}{\Gamma(t+\alpha+1)\Gamma(\beta+1)}.$$
Then,
$$\frac{\partial^2}{\partial\beta^2}\log f_{x,\alpha}(t,\beta)=-\psi^\prime(\beta+1)<0\;\;\textrm{and}\;\;\frac{\partial^2}{\partial\alpha^2}\log f_{x,\alpha}(t,\beta)=-\psi^\prime(t+\alpha+1)<0,$$
where $\psi$ is the digamma function. This implies that the functions $\beta\mapsto f_{x,\alpha}(t,\beta)$ and $\alpha\mapsto f_{x,\alpha}(t,\beta)$ are log-concave on $(-1,\infty)$ and consequently the functions $\beta\mapsto\mu(x,\beta,\alpha)$ and $\alpha\mapsto\mu(x,\beta,\alpha)$ are log-concave on $(-1,\infty)$ by means of Corollary 3.5 in \cite{BL}.

\textbf{d.} Again, by using the Rogers--H\"older--Riesz inequality, we get for all $\beta_1,\beta_2>-1$ and $\lambda\in[0,1]$ we have
\begin{equation}
\begin{split}
H(\lambda\beta_1+(1-\lambda)\beta_2)&=\int_0^\infty\frac{x^{t+\alpha}t^{\lambda\beta_1+(1-\lambda)\beta_2}}{\Gamma(t+\alpha+1)}dt\\
&=\int_0^\infty\left[\frac{x^{t+\alpha}t^{\beta_1}}{\Gamma(t+\alpha+1)}\right]^\lambda\left[\frac{x^{t+\alpha}t^{\beta_2}}{\Gamma(t+\alpha+1)}\right]^{1-\lambda} dt\\
&\leq\left[\int_0^\infty\frac{x^{t+\alpha}t^{\beta_1}}{\Gamma(t+\alpha+1)}dt\right]^\lambda\left[\int_0^\infty\frac{x^{t+\alpha}t^{\beta_2}}{\Gamma(t+\alpha+1)}dt\right]^{1-\lambda}\\
&=[H(\beta_1)]^\lambda[H(\beta_2)]^{1-\lambda}.
\end{split}
\end{equation}

\textbf{f.} From the Ramanujan identity \cite[p. 196]{HA}
\begin{equation}
e^x-\nu(x)=\int_0^\infty e^{xe^t}\frac{dt}{t^2+\pi^2},
\end{equation}
we obtain
\begin{equation}
e^x-\nu(x)=\int_0^\infty e^{xt}\frac{dt}{t(\log^2(t)+\pi^2)}.
\end{equation}
 Therefore, we deduce that all prerequisites of the Bernstein Characterization
Theorem for the complete monotone functions are fulfilled, that is, the function $x\mapsto e^{-x}-\nu(-x) $ is completely monotonic on $(0,\infty).$
Which evidently completes the proof of the Theorem \ref{T1}.
\end{proof}

Now, we provide new inequalities for the Volterra function $\mu(x,\beta,\alpha).$

\begin{corollary}\label{TT90}The following inequalities hold true:\\
\noindent \textbf{a.} For all $x,y>0$ and $\alpha\geq0$ we have
\begin{equation}\label{i1}
\mu(\sqrt{xy},\beta,\alpha)\leq\sqrt{\mu(x,\beta,\alpha)\mu(y,\beta,\alpha)}.
\end{equation}
\noindent \textbf{b.}  The following Tur\'an type inequality
\begin{equation}\label{TU}
\mu^2(x,\beta,\alpha+1)-\mu(x,\beta,\alpha) \mu(x,\beta,\alpha+2)\geq0,
\end{equation}
holds true for all $\alpha,\beta>-1$ and $x>0.$\\
\noindent \textbf{c.} The following Tur\'an type inequalities
\begin{equation}\label{Y9}
0\leq\mu^2(x,\beta+1,\alpha)-\mu(x,\beta,\alpha)\mu(x,\beta+2,\alpha)\leq\frac{\mu^2(x,\beta+1,\alpha)}{\beta+2},
\end{equation}
hold true for all $\alpha,\beta>-1,\;x>0.$\\
\noindent \textbf{d.} The following inequality
\begin{equation}\label{MAL}
\nu(-x)\nu(-y)+\nu(-x-y)\leq e^{-x}\nu(-y)+e^{-y}\mu(-x),
\end{equation}
holds true for all $x,y>0.$
\end{corollary}
\begin{proof}The inequality (\ref{i1}) is an immediate consequence of the fact that the function $x\mapsto \mu(x,\beta,\alpha)$ is geometrically convex on $(0,\infty).$
Now, focus on the Tur\'an type inequality (\ref{TU}). Since the function $\alpha\mapsto \mu(x,\beta,\alpha)$ is log--concave on $(-1,\infty)$ for  all $x>0,$  it follows that for all $\alpha_1,\alpha_2\geq0,\;\lambda\in[0,1]$ and $x>0,$ we have
$$\mu(x,\beta,\lambda\alpha_1+(1-\lambda)\alpha_2)\geq\left[\mu(x,\beta,\alpha_1)\right]^\lambda\left[\mu(x,\beta,\alpha_2)\right]^{1-\lambda}.$$
Choosing $\alpha_1=\alpha, \alpha_2=\alpha+2$ and $\lambda=\frac{1}{2},$ the above inequality reduces to the Tur\'an type inequality (\ref{TU}).  %µFrom this Theorem part c.  it is clear that the function
%$$\mathcal{V}_x(\alpha)=\frac{\nu(x,\alpha)}{\nu(x)}$$
%maps $(0,\infty)$ into $(0,1)$ and it is completely monotonic on $(0,\infty)$ for all $x>0.$  On the other hand, according to Kimberling \cite{KI}  if a function $f,$ defined
%on $(0,\infty),$ is continuous and completely monotonic and maps $(0,\infty)$ into $(0,1),$ then
%\begin{equation}\label{kim}
%f(x)f(y)\leq f(x+y).
%\end{equation}
%Therefore we conclude the asserted inequality (\ref{i3}).
 The Tur\'an type inequalities (\ref{Y9}) follows immediately using the fact that the function $H(\beta)$ is log--convex and the function $\mu(x,\beta,\alpha)$ is log--concave on $(-1,\infty)$. It remains to prove (\ref{MAL}). But using the fact that the function $x\mapsto e^{-x}-\nu(-x)$ is completely monotonic on $(0,\infty),$ we have
\begin{equation}
e^{-x}-\nu(-x)\leq \int_{-\infty}^\infty\frac{dt}{t^2+\pi^2}=1.
\end{equation}
Therefore, the function $x\mapsto e^{-x}-\nu(-x)$ maps $(0,\infty)$ into $(0,1),$ and consequently the inequality (\ref{MAL}) holds true, by the Kimberling's result \cite{KI}: if a function $f,$ defined
on $(0,\infty),$ is continuous and completely monotonic and maps $(0,\infty)$ into $(0,1),$ then
\begin{equation}\label{kim}
f(x)f(y)\leq f(x+y).
\end{equation}
The proof of Theorem \ref{TT90} is complete.
\end{proof}

Now, we define the function $F_{x,\alpha}(\beta)$ by
\begin{equation}
F_{x,\alpha}(\beta)=\Gamma(\beta+1)\mu(x,\beta,\alpha),
\end{equation}
where $\alpha,\beta>-1$ and $x>0.$

\begin{theorem}\label{TT3}For all odd integers $n\geq1$ and real numbers $x> 0$ and $\alpha>-1,$ the following Tur\'an type inequality
\begin{equation}\label{EEE}
\Delta_n^{(x,\alpha)}(\beta)=\left(\frac{\partial^{n-1}F_{x,\alpha}(\beta)}{\partial \beta^{n-1}}\right)\left(\frac{\partial^{n+1}F_{x,\alpha}(\beta)}{\partial \beta^{n+1}}\right)-\left(\frac{\partial^{n}F_{x,\alpha}(\beta)}{\partial \beta^{n}}\right)^2\geq0,
\end{equation}
holds true for all $\beta>-1.$
\end{theorem}
\begin{proof}Let  $n\geq1$ be a odd integers and real numbers $x> 0$ and $\alpha>-1.$ Then we get
\begin{equation}\label{k!}
\begin{split}
\Delta_n^{(x,\alpha)}(\beta)&=\int_0^\infty\frac{x^{t+\alpha}t^\beta\;\log^{n-1}(t)}{\Gamma(t+\alpha+1)}dt\int_0^\infty\frac{x^{t+\alpha}t^\beta\;\log^{n+1}(t)}{\Gamma(t+\alpha+1)}dt-\left(\int_0^\infty\frac{x^{t+\alpha}t^\beta\;\log^{n}(t)}{\Gamma(t+\alpha+1)}dt\right)^2\\
&=\int_0^\infty\int_0^\infty \frac{x^{t+s+2\alpha}(ts)^\beta}{\Gamma(t+\alpha+1)\Gamma(s+\alpha+1)}\left[\log^{n-1}(t)\log^{n+1}(s)-\log^{n}(t)\log^{n}(s)\right]dsdt\\
&=\int_0^\infty\int_0^\infty \frac{x^{t+s+2\alpha}(ts)^\beta}{\Gamma(t+\alpha+1)\Gamma(s+\alpha+1)}\left[\log(t)\log(s)\right]^{n-1}\left[\log^{2}(t)-\log(t)\log(s)\right]dsdt\\
&=\frac{1}{2}\int_0^\infty\int_0^\infty \frac{x^{t+s+2\alpha}(ts)^\beta}{\Gamma(t+\alpha+1)\Gamma(s+\alpha+1)}\left[\log(t)\log(s)\right]^{n-1}\left[\log(t)-\log(s)\right]^2dsdt\geq0.
\end{split}
\end{equation}

It is important to mention here that there is another proof of inequality (\ref{EEE}). Namely, for all odd integers $n\geq1$ and real numbers $x> 0$ and $\alpha>-1,$ we apply the  Cauchy--Bunyakovskii inequality and find
\begin{equation}
\begin{split}
\left(\frac{\partial^{n}F_{x,\alpha}(\beta)}{\partial \beta^{n-1}}\right)^2&=\left(\int_0^\infty \frac{x^{t+\alpha}t^\beta\;\log^{n}(t)}{\Gamma(t+\alpha+1)}dt\right)^2\\
&=\left(\int_0^\infty \left[\frac{x^{t+\alpha}t^\beta\;\log^{n-1}(t)}{\Gamma(t+\alpha+1)}\right]^{\frac{1}{2}}\left[\frac{x^{t+\alpha}t^\beta\;\log^{n+1}(t)}{\Gamma(t+\alpha+1)}\right]^{\frac{1}{2}}dt\right)^2\\
&\leq\left[\int_0^\infty \frac{x^{t+\alpha}t^\beta\;\log^{n-1}(t)}{\Gamma(t+\alpha+1)}\right]\left[\int_0^\infty \frac{x^{t+\alpha}t^\beta\;\log^{n+1}(t)}{\Gamma(t+\alpha+1)}\right]\\
&=\left(\frac{\partial^{n-1}F_{x,\alpha}(\beta)}{\partial \beta^{n-1}}\right)\left(\frac{\partial^{n+1}F_{x,\alpha}(\beta)}{\partial \beta^{n+1}}\right).
\end{split}
\end{equation}
The proof of the Theorem \ref{TT3} is complete.
\end{proof}

\begin{corollary}For all odd integers $n\geq1$ and real numbers $x> 0$ and $\alpha>-1,$  the function $\beta\mapsto\Delta_n^{(x,\alpha)}(\beta)$ is convex on $(-1,\infty).$ Moreover, the following Schur--type functional inequality
\begin{equation}\label{SCHUR}
\begin{split}
(\beta_1-\beta_2)(\beta_1-\beta_3)\Delta_n^{(x,\alpha)}(\beta_1)&+(\beta_2-\beta_1)(\beta_2-\beta_3)\Delta_n^{(x,\alpha)}(\beta_2)\\&+(\beta_3-\beta_1)(\beta_3-\beta_2)\Delta_n^{(x,\alpha)}(\beta_2)\geq0,
\end{split}
\end{equation}
is valid for all $\beta_1,\beta_2,\beta_3>-1.$
\end{corollary}
\begin{proof}We set $n=2k-1$ with $k\geq1.$  Differentiation gives for $\beta>-1$
\begin{equation*}
\frac{\partial^2\Delta_{2k-1}^{(x,\alpha)}(\beta)}{\partial\beta^2}=\left(\frac{\partial^{2k-2}F_{x,\alpha}(\beta)}{\partial \beta^{2k-2}}\right)\left(\frac{\partial^{n+1}F_{x,\alpha}(\beta)}{\partial \beta^{2k+2}}\right)-\left(\frac{\partial^{2k}F_{x,\alpha}(\beta)}{\partial \beta^{2k}}\right)^2.
\end{equation*}
So, applying the Cauchy--Bunyakovskii inequality yields
$$\frac{\partial^2\Delta_{2k-1}^{(x,\alpha)}(\beta)}{\partial\beta^2}\geq0.$$
Since $\Delta_{n}^{(x,\alpha)}(\beta)$ is nonnegative and convex on $(-1,\infty),$ we conclude that the Schur--type functional inequality (\ref{SCHUR}) holds true (see \cite{WR}).
\end{proof}

\section{Monotonicity Properties and functional inequalities for the incomplete Volterra function}

\begin{theorem} \label{T6}Let $\alpha>z^*,\;0<x<1$ and $s>0.$ Then the function $\beta\mapsto G_\beta(x,\alpha,s)$
is log--concave on $(-1,\infty),$ where $z^*\simeq 1.461632144...$ is the abscissa of the minimum of the $\Gamma$ function. In particular, the Tur\'an type inequality
\begin{equation}
\left(G_{\beta+1}(x,\alpha,s)\right)^2-G_{\beta+1}(x,\alpha,s)G_{\beta+2}(x,\alpha,s)\geq0,
\end{equation}
holds for all $s>0,\;0<x<1$ and $\alpha>z^*.$
\end{theorem}
\begin{proof}%Let $$\phi_\beta(x,\alpha,s)=\Gamma(\beta+1)\psi_\beta(x,\alpha,s).$$
We show that
\begin{equation}
\frac{\partial^2}{\partial\beta^2}\log G_{\beta+1}(x,\alpha,s)\leq0,
\end{equation}
for $\alpha,\beta>-1,\;0<x<1$ and $s>z^*.$ Let $K(\beta)=\frac{\partial^2\log H(\beta)}{\partial\beta^2},$ with $H(\beta)$ is defined in Theorem \ref{T1}. Then we have
\begin{equation}\label{MMMMm}
\begin{split}
\mu^2(x,\beta,\alpha,s)\frac{\partial^2}{\partial\beta^2}\log G_\beta(x,\alpha,s)&=\int_s^\infty\frac{x^{t+\alpha}t^\beta}{\Gamma(t+\alpha+1)}dt\int_s^\infty\frac{x^{t+\alpha}t^\beta\;\log^2(t)}{\Gamma(t+\alpha+1)}dt\\
&-\left(\int_s^\infty\frac{x^{t+\alpha}t^\beta\;\log(t)}{\Gamma(t+\alpha+1)}dt\right)^2-K(\beta)\left(\int_s^\infty\frac{x^{t+\alpha}t^\beta}{\Gamma(t+\alpha+1)}dt\right)^2\\
&=U_\beta(x,\alpha,s).
\end{split}
\end{equation}
We derive
\begin{equation}\label{MALL}
\begin{split}
\frac{\Gamma(s+\alpha+1)}{x^{s+\alpha}s^\beta}\frac{\partial}{\partial s}U_\beta(x,\alpha,s)&=-\log^2(s)\int_s^\infty\frac{x^{t+\alpha}t^\beta}{\Gamma(t+\alpha+1)}dt-\int_s^\infty\frac{x^{t+\alpha}t^\beta\;\log^2(t)}{\Gamma(t+\alpha+1)}dt\\
&+2\log(s)\int_s^\infty\frac{x^{t+\alpha}t^\beta\;\log(t)}{\Gamma(t+\alpha+1)}dt+2K(\beta)\int_s^\infty\frac{x^{t+\alpha}t^\beta}{\Gamma(t+\alpha+1)}dt.
\end{split}
\end{equation}
We denote the expression on the right--hand side of (\ref{MALL}) by $V_\beta(x,\alpha,s).$ Then we have
\begin{equation}
\begin{split}
\frac{\partial}{2\partial s}V_\beta(x,\alpha,s)&=\frac{1}{s}\int_s^\infty\frac{x^{t+\alpha}t^\beta\left(\log(t)-\log(s)\right)}{\Gamma(t+\alpha+1)}dt-K(\beta)\frac{x^{s+\alpha}s^\beta}{\Gamma(s+\alpha+1)}\\
&=x^\alpha s^\beta\left(\int_1^\infty\frac{x^{st}t^\beta\log(t)}{\Gamma(st+\alpha+1)}dt-\frac{K(\beta)x^s}{\Gamma(s+\alpha+1)}\right)\\
&=\frac{x^{\alpha+s} s^\beta}{\Gamma(s+\alpha+1)} W_\beta(x,\alpha,s),
\end{split}
\end{equation}
where $ W_\beta(x,s)$ is defined by
\begin{equation}\label{yyy} W_\beta(x,\alpha,s)=\int_1^\infty t^\beta\log(t)\omega_\beta(x,\alpha,s)dt-K(\beta),\;\textrm{with}\;\;\omega_\beta(x,\alpha,s,t)=\frac{x^{s(t-1)}\Gamma(s+\alpha+1)}{\Gamma(st+\alpha+1)}.\end{equation}
Thus, we have
\begin{equation}\label{mehrezm}
\begin{split}
\frac{\partial}{\partial s}\omega_\beta(x,\alpha,s)&=\frac{(t-1)\log(x)x^{s(t-1)}\Gamma(s+\alpha+1)}{\Gamma(st+\alpha+1)}\\
&+\frac{x^{s(t-1)}\Gamma(s+\alpha+1)\left[\psi(s+\alpha+1)-t\psi(st+\alpha+1)\right]}{\Gamma(st+\alpha+1)}.
\end{split}
\end{equation}
Since the function $\psi$ is increasing on $(0,\infty),$ we obtain
\begin{equation}
\begin{split}
\frac{\partial}{\partial s}\omega_\beta(x,\alpha,s)&\leq\frac{(t-1)\log(x)x^{s(t-1)}\Gamma(s+\alpha+1)}{\Gamma(st+\alpha+1)}\\
&+\frac{x^{s(t-1)}\Gamma(s+\alpha+1)\left[\psi(s+\alpha+1)-\psi(st+\alpha+1)\right]}{\Gamma(st+\alpha+1)}\\
&\leq0,
\end{split}
\end{equation}
for all $t>1$ and $0<x<1.$ Thus implies that $\frac{\partial}{\partial s}W_\beta(x,\alpha,s)<0.$
Moreover, from (\ref{yyy}) we derive
\begin{equation}\label{bnbb}
\lim_{s\rightarrow0}W_\beta(x,\alpha,s)=\infty.
\end{equation}
On the other hand, by using the fact that the function $s\mapsto\Gamma(s)$ is increasing on $(z^*,\infty),$ we find that
\begin{equation*}
W_\beta(x,\alpha,s)\leq\int_1^\infty x^{s(t-1)}t^\beta\log(t)dt-K(\beta).
\end{equation*}
Then,
\begin{equation}\label{bnb}
\lim_{s\rightarrow\infty}W_\beta(x,\alpha,s)\leq-K(\beta).
\end{equation}
Furthermore, since the function $K(\beta)\geq0$ for all $\beta>-1,$ by means of Theorem \ref{T1},  and in view of (\ref{bnbb}) and (\ref{bnb}) we deduce that there exists a positive number $s_1$ such that $W_\beta$ is positive on $(0,x_1)$ and negative on $(x_1,\infty).$ This implies that the function $V_\beta$ is increasing on $(0,x_1)$ and decreasing on $(x_1,\infty).$ Furthermore, we have
\begin{equation}\label{MMM}
\begin{split}
\lim_{s\rightarrow\infty}V_\beta(x,\alpha,s)&=\lim_{s\rightarrow\infty}\Big(-\log^2(s)\int_s^\infty\frac{x^{t+\alpha}t^\beta}{\Gamma(t+\alpha+1)}dt-\int_s^\infty\frac{x^{t+\alpha}t^\beta\;\log^2(t)}{\Gamma(t+\alpha+1)}dt\\
&+2\log(s)\int_s^\infty\frac{x^{t+\alpha}t^\beta\;\log(t)}{\Gamma(t+\alpha+1)}dt+2K(\beta)\int_s^\infty\frac{x^{t+\alpha}t^\beta}{\Gamma(t+\alpha+1)}dt\Big)\\
&=\lim_{s\rightarrow\infty}\left(-\log^2(s)\int_s^\infty\frac{x^{t+\alpha}t^\beta}{\Gamma(t+\alpha+1)}dt+2\log(s)\int_s^\infty\frac{x^{t+\alpha}t^\beta\;\log(t)}{\Gamma(t+\alpha+1)}dt\right)\\
&=\lim_{s\rightarrow\infty}(I_1(s)+I_2(s)),\;(say.)
\end{split}
\end{equation}
Hospital's rule leads to
\begin{equation}\label{MM}
\begin{split}
\lim_{s\rightarrow\infty}I_1(s)&=\lim_{s\rightarrow\infty}\left(\int_0^s\frac{x^{t+\alpha}t^\beta}{\Gamma(t+\alpha+1)}dt-\int_0^\infty\frac{x^{t+\alpha}t^\beta}{\Gamma(t+\alpha+1)}dt\Big/\log^{-2}(s)\right)\\
&=-\lim_{s\rightarrow\infty}\frac{x^\alpha s^{\beta+1}\log^3(s)e^{s\log(x)}}{2\Gamma(s+\alpha+1)}\\
&=0,
\end{split}
\end{equation}
and
\begin{equation}\label{MMMMmm}
\begin{split}
\lim_{s\rightarrow\infty}I_2(s)&=\lim_{s\rightarrow\infty}\left(\int_0^\infty\frac{x^{t+\alpha}t^\beta\log(t)}{\Gamma(t+\alpha+1)}dt-\int_0^s\frac{x^{t+\alpha}t^\beta\log(t)}{\Gamma(t+\alpha+1)}dt\Big/\log^{-1}(s)\right)\\
&=\lim_{s\rightarrow\infty}\frac{x^\alpha s^{\beta+1}\log^3(s)e^{s\log(x)}}{\Gamma(s+\alpha+1)}\\
&=0.
\end{split}
\end{equation}
From (\ref{MMM}) (\ref{MM})and (\ref{MMMMmm}) we find that $$\lim_{s\rightarrow\infty}V_\beta(x,\alpha,s)=0.$$
On the other hand, we have
$$\lim_{s\rightarrow0}V_\beta(x,\alpha,s)=-\infty.$$
So, we obtain that there exists a positive number $s_2$ such that $V_\beta$ is negative on $(0,s_2)$ and positive on $(s_2,\infty).$ So, by (\ref{MALL}) we deduce that the function $U_\beta$ is decreasing on $(0,s_2)$ and increasing on $(s_2,\infty).$ Moreover,
$$U_\beta(x,\alpha,0)=\lim_{s\rightarrow\infty}U_\beta(x,\alpha,s)=0.$$
Then $U_\beta(x,\alpha,s)\leq 0$ and consequently, the function $\beta\mapsto G_\beta(x,\alpha,s)$ is log-concave on $(-1,\infty)$ for all $s>0,\alpha>z^*$ and $0<x<1.$ % Finally, by using the fact that the function $\beta\mapsto\frac{1}{\Gamma(\beta+1)}$ is log-concave on $(-1,\infty),$ we get that the fucntion $\beta\mapsto\psi_\beta(x,\alpha,s)$ is also log-concave on $(-1,\infty),$ as a product of two log-concave functions.
The proof of Theorem \ref{T6} is complete.
\end{proof}

Now we consider some mean--value inequalities. For more information on power means and their applications see e.g. \cite{MBV}, \cite{Sit2}.

The power mean of order $t$ is defined by
\begin{equation*}
\begin{split}
M_t(x_1,x_2;\lambda)&=\left(\lambda x_1^t+(1-\lambda)x_2^t\right)^{\frac{1}{t}}\;(t\neq0,\;0<\lambda<1),\\
M_0(x_1,x_2;\lambda)&=x_1^\lambda x_2^{1-\lambda}\\
M_{-\infty}(x_1,x_2;\lambda)&=\min(x_1,x_2),\;\;
M_{\infty}(x_1,x_2;\lambda)=\max(x_1,x_2).
\end{split}
\end{equation*}
\begin{corollary}Let $\beta_1,\beta_2>-1,\;\beta_1\neq\beta_2$ and $\lambda\in(0,1).$ Then, the following inequality
\begin{equation}\label{SARRA1}
M_r\left(G_{\beta_1}(x,\alpha,s),G_{\beta_2}(x,\alpha,s);\lambda\right)\leq G_{\lambda\beta_1+(1-\lambda)\beta_2}(x,\alpha,s)
\end{equation}
is valid for all $x>0,s\geq0,\alpha>-1$ if and only if $r\leq0.$
\end{corollary}

\begin{proof}By using the fact that the function $G_\beta$ is log--concave on $(-1,\infty),$ we conclude that the inequality (\ref{SARRA1}) with $r=0$  holds for all $x>0,\alpha>-1.$ We suppose that there exists a real number $r>0$ such that the inequality (\ref{SARRA1}) holds true. This implies that
\begin{equation}\label{ZAR}
\lambda\left[\frac{G_{\beta_1}(x,\alpha,s)}{G_{\lambda\beta_1+(1-\lambda)\beta_2}(x,\alpha,s)}\right]^r+(1-\lambda)\left[\frac{G_{\beta_2}(x,\alpha,s)}{G_{\lambda\beta_1+(1-\lambda)\beta_2}(x,\alpha,s)}\right]^r\leq1.
\end{equation}
From the l'Hospital's rule we obtain
\begin{equation}
\begin{split}
\lim_{s\rightarrow\infty}\frac{G_{\beta_1}(x,\alpha,s)}{G_{\lambda\beta_1+(1-\lambda)\beta_2}(x,\alpha,s)}&=\lim_{s\rightarrow\infty}\frac{\mu(x,\lambda\beta_1+(1-\lambda)\beta_2,\alpha)\Gamma(\lambda\beta_1+(1-\lambda)\beta_2+1)s^{(1-\lambda)(\beta_1-\beta_2)}}{\mu(x,\beta_1,\alpha)\Gamma(\beta_1+1)}\\
&= \left\{ \begin{array}{ll}
\infty, & \textrm{if $\beta_1>\beta_2$}\\
0, & \textrm{if $\beta_1<\beta_2$}
\end{array} \right.
\end{split}
\end{equation}
and
\begin{equation}
\begin{split}
\lim_{s\rightarrow0}\frac{G^*_{\beta_1}(x,\alpha,s)}{G^*_{\lambda\beta_1+(1-\lambda)\beta_2}(x,\alpha,s)}&=\lim_{s\rightarrow0}\frac{\mu(x,\lambda\beta_1+(1-\lambda)\beta_2,\alpha)\Gamma(\lambda\beta_1+(1-\lambda)\beta_2+1)s^{\lambda(\beta_2-\beta_1)}}{\mu(x,\beta_2,\alpha)\Gamma(\beta_2+1)}\\
&= \left\{ \begin{array}{ll}
0, & \textrm{if $\beta_1>\beta_2$}\\
\infty, & \textrm{if $\beta_1<\beta_2$}
\end{array} \right.
\end{split}
\end{equation}
This implies that the left--hand side of the inequality (\ref{ZAR}) tends to $\infty$ if $s$ tends to $0.$ It is a  contradiction, so the inequality (\ref{SARRA1}) is valid for all  $x>0, s\geq0, \alpha>-1$ and $r\leq0.$
\end{proof}

\begin{theorem}\label{T7} Suppose that conditions of Theorem \ref{T6} are satisfied. Then the function $G^*_\beta$ is log--concave on $(-1,\infty).$ In particular, the Tur\'an type inequality
\begin{equation}
\left(G_{\beta+1}^*(x,\alpha,s)\right)^2-G_\beta^*(x,\alpha,s)G_{\beta+2}^*(x,\alpha,s)\geq0,
\end{equation}
holds for all $s>0,\;0<x<1$ and $\alpha>z^*.$
\end{theorem}
\begin{proof} We prove that
\begin{equation}\label{SIT5}
\frac{\partial^2}{\partial\beta^2}\log G_\beta^*(x,\alpha,s)\leq0,
\end{equation}
for $\alpha,\beta>-1,\;0<x<1$ and $s>z^*.$ Thus,
\begin{equation}\label{MMMM}
\begin{split}
\mu^2(x,\beta,\alpha)\frac{\partial^2}{\partial\beta^2}\log G_\beta^*(x,\alpha,s)&=\int_0^s\frac{x^{t+\alpha}t^\beta}{\Gamma(t+\alpha+1)}dt\int_0^s\frac{x^{t+\alpha}t^\beta\;\log^2(t)}{\Gamma(t+\alpha+1)}dt\\
&-\left(\int_0^s\frac{x^{t+\alpha}t^\beta\;\log(t)}{\Gamma(t+\alpha+1)}dt\right)^2-K(\beta)\left(\int_0^s\frac{x^{t+\alpha}t^\beta}{\Gamma(t+\alpha+1)}dt\right)^2\\
&=U_\beta^*(x,\alpha,s),\;\textrm{say.}
\end{split}
\end{equation}
Then we have
\begin{equation}\label{MAILL}
\begin{split}
\frac{\Gamma(s+\alpha+1)}{x^{s+\alpha}s^\beta}\frac{\partial}{\partial s}U_\beta^*(x,\alpha,s)&=\log^2(s)\int_0^s\frac{x^{t+\alpha}t^\beta}{\Gamma(t+\alpha+1)}dt+\int_0^s\frac{x^{t+\alpha}t^\beta\;\log^2(t)}{\Gamma(t+\alpha+1)}dt\\
&-2\log(s)\int_0^s\frac{x^{t+\alpha}t^\beta\;\log(t)}{\Gamma(t+\alpha+1)}dt-2K(\beta)\int_0^s\frac{x^{t+\alpha}t^\beta}{\Gamma(t+\alpha+1)}dt\\
&=V_\beta^*(x,\alpha,s).
\end{split}
\end{equation}
\begin{equation}
\begin{split}
\frac{\partial}{2\partial s}V_\beta^*(x,\alpha,s)&=\frac{1}{s}\int_0^s\frac{x^{t+\alpha}t^\beta\left(\log(s)-\log(t)\right)}{\Gamma(t+\alpha+1)}dt-K(\beta)\frac{x^{s+\alpha}s^\beta}{\Gamma(s+\alpha+1)}\\
&=x^\alpha s^\beta\left(-\int_0^1\frac{x^{st}t^\beta\log(t)}{\Gamma(st+\alpha+1)}dt-K(\beta)\frac{x^{s}s^\beta}{\Gamma(s+\alpha+1)}\right)\\
&=\frac{x^{\alpha+s} s^\beta}{\Gamma(s+\alpha+1)} W_\beta^*(x,\alpha,s),
\end{split}
\end{equation}
with
$$
W_\beta^*(x,\alpha,s)=-\int_0^1 t^\beta\log(t)\omega_\beta(x,\alpha,t,s)dt-K(\beta).
$$
In our case, from (\ref{mehrezm}) we have
\begin{equation*}
\begin{split}
\frac{\partial}{\partial s}\omega_\beta(x,\alpha,s)\geq\frac{x^{s(t-1)}\Gamma(s+\alpha+1)\left[\psi(s+\alpha+1)-t\psi(st+\alpha+1)\right]}{\Gamma(st+\alpha+1)}.
\end{split}
\end{equation*}
So, for all $0<t<1$ and $\alpha>z^*$ we obtain that
\begin{equation*}
\begin{split}
\psi(s+\alpha+1)-t\psi(st+\alpha+1)&\geq\psi(s+\alpha+1)-\psi(st+\alpha+1)\\&\geq\psi(s+\alpha+1)-\psi(s+\alpha+1)\\&=0.
\end{split}
\end{equation*}
Consequently, the function $W_\beta^*$ is increasing on $(0,\infty).$ Moreover, we have
$$W_\beta^*(x,\alpha,s)\geq-\int_0^1x^{s(t-1)}t^\beta\log(t)dt-K(\beta).$$
This implies that
$$\lim_{s\rightarrow\infty}W_\beta^*(x,\alpha,s)=\infty.$$
In addition, we have
\begin{equation}
\lim_{s\rightarrow0}W_\beta^*(x,\alpha,s)=\frac{1}{(\beta+1)^2}-K(\beta).
\end{equation}
We assume that the function $W_\beta^*(x,\alpha,s)$ attains only positive values on $(0,\infty)$ (i.e $\frac{1}{(\beta+1)^2}-K(\beta)\geq0.$) This implies that the function $V_\beta^*(x,\alpha,s)$ is increasing on $(0,\infty).$ In addition, we have
\begin{equation}\label{SIT0}
\lim_{s\rightarrow\infty}V_\beta^*(x,\alpha,s)=\infty.
\end{equation}
Again, by using the l'Hospital's rule we obtain
\begin{equation}\label{SIT1}
\begin{split}
\lim_{s\rightarrow0}V_\beta^*(x,\alpha,s)&=\lim_{s\rightarrow0}\left(\log^2(s)\int_0^s\frac{x^{t+\alpha}t^\beta}{\Gamma(t+\alpha+1)}dt-2\log(s)\int_0^s\frac{x^{t+\alpha}t^\beta\;\log(t)}{\Gamma(t+\alpha+1)}dt\right)\\
&=\lim_{s\rightarrow0}\left(\left[\int_0^s\frac{x^{t+\alpha}t^\beta}{\Gamma(t+\alpha+1)}dt\Big/\log^{-2}(s)\right]-2\left[\int_0^s\frac{x^{t+\alpha}t^\beta\log(t)}{\Gamma(t+\alpha+1)}dt\Big/\log^{-1}(s)\right]\right)\\
&=\lim_{s\rightarrow0}\left(-\frac{s^{\beta+1}x^{s+\alpha}\log^3(s)}{2\Gamma(s+\alpha+1)}+\frac{2s^{\beta+1}x^{s+\alpha}\log^3(s)}{2\Gamma(s+\alpha+1)}\right)\\
&=0.
\end{split}
\end{equation}
From (\ref{MAILL}), (\ref{SIT0}), (\ref{SIT1}) and the monotonicity of $V_\beta$ we obtain that the function $U_\beta$ is increasing on $(0,\infty)$ such that
\begin{equation}\label{SIT3}
\lim_{s\rightarrow0}U_\beta^*(x,\alpha,s)=\lim_{s\rightarrow\infty}U_\beta^*(x,\alpha,s)=0.
\end{equation}
We receive a contradiction,  it implies that $$\frac{1}{(\beta+1)^2}-K(\beta)\leq0.$$ Consequently,  this implies that there exists a positive number $s_3$ such that $W_\beta^*(x,\alpha,s)$ is negative on $(0,s_3)$ and positive on $(s_3,\infty)$ and consequently the function $V_\beta^*(x,\alpha,s)$ is decreasing on $(0,s_3)$ and increasing on $(s_3,\infty).$ So, by using (\ref{SIT0}), (\ref{SIT1}) and  monotonicity of the function $V_\beta^*(x,\alpha,s)$ we deduce that there exists a positive number $s_4$ such that the function $V_\beta^*(x,\alpha,s)$ is negative on $(0,s_4)$ and positive on $(s_4,\infty).$ This yields that the function $U_\beta^*(x,\alpha,s)$ is decreasing on $(0,s_4)$ and increasing on $(s_4,\infty).$ So, using the monotonicity of the function $U_\beta^*(x,\alpha,s)$ and (\ref{SIT3}) we deduce
$U_\beta^*(x,\alpha,s)\leq0.$ This proves (\ref{SIT5}), which evidently completes the proof of the Theorem \ref{T7}.
\end{proof}

\begin{corollary}Let $\beta_1,\beta_2>-1,\;\beta_1\neq\beta_2$ and $\lambda\in(0,1).$ Then the following inequality
\begin{equation}\label{SARRA}
M_r\left(G_{\beta_1}^*(x,\alpha,s),G_{\beta_2}^*(x,\alpha,s);\lambda\right)\leq G_{\lambda\beta_1+(1-\lambda)\beta_2}^*(x,\alpha,s)
\end{equation}
is valid for all $x>0,s\geq0,\alpha>-1$ if and only if $r\leq0.$
\end{corollary}
\begin{proof}Since the function $G^*_\beta$ is log--concave on $(-1,\infty),$ we conclude that the inequality (\ref{SARRA}) with $r=0$ is valid for all $x>0,\alpha>-1.$ We proved that there exists a real number $r>0$ such that the inequality (\ref{SARRA}) holds true. This implies that
\begin{equation}
\lambda\left[\frac{G^*_{\beta_1}(x,\alpha,s)}{G^*_{\lambda\beta_1+(1-\lambda)\beta_2}(x,\alpha,s)}\right]^r+(1-\lambda)\left[\frac{G^*_{\beta_2}(x,\alpha,s)}{G^*_{\lambda\beta_1+(1-\lambda)\beta_2}(x,\alpha,s)}\right]^r\leq1.
\end{equation}
l'Hospital's rule leads to
\begin{equation}
\begin{split}
\lim_{s\rightarrow0}\frac{G^*_{\beta_1}(x,\alpha,s)}{G^*_{\lambda\beta_1+(1-\lambda)\beta_2}(x,\alpha,s)}&=\lim_{s\rightarrow0}\frac{\mu(x,\lambda\beta_1+(1-\lambda)\beta_2,\alpha)\Gamma(\lambda\beta_1+(1-\lambda)\beta_2+1)s^{(1-\lambda)(\beta_1-\beta_2)}}{\mu(x,\beta_1,\alpha)\Gamma(\beta_1+1)}\\
&= \left\{ \begin{array}{ll}
0, & \textrm{if $\beta_1>\beta_2$}\\
\infty, & \textrm{if $\beta_1<\beta_2$}
\end{array} \right.
\end{split}
\end{equation}
and
\begin{equation}
\begin{split}
\lim_{s\rightarrow0}\frac{G^*_{\beta_1}(x,\alpha,s)}{G^*_{\lambda\beta_1+(1-\lambda)\beta_2}(x,\alpha,s)}&=\lim_{s\rightarrow0}\frac{\Gamma(\lambda\beta_1+(1-\lambda)\beta_2+1)s^{\lambda(\beta_2-\beta_1)}}{\Gamma(\beta_2+1)}\\
&= \left\{ \begin{array}{ll}
\infty, & \textrm{if $\beta_1>\beta_2$}\\
0, & \textrm{if $\beta_1<\beta_2$}
\end{array} \right.
\end{split}
\end{equation}
Then the left--hand side of inequality tends to $\infty$ if $s$ tends to $0.$ A contradiction! This implies that the inequality (\ref{SARRA}) for all  $x>0,s\geq0,\alpha>-1$ and $r\leq0.$
\end{proof}
\begin{theorem}\label{TTT7} Let $\alpha,\beta>-1$ and $x>0.$ Then the function
$\beta\mapsto G_\beta^*(x,\alpha,s)$ is decreasing on $(-1,\infty).$ Moreover, the following inequality
\begin{equation}\label{0000}
G_{\lambda\beta_1+(1-\lambda)\beta_2}^*(x,\alpha,s)\leq M_\kappa\left(G_{\beta_1}^*(x,\alpha,s),G_{\beta_2}^*(x,\alpha,s);\lambda\right)
\end{equation}
\end{theorem}
are valid for all real numbers $\alpha,\beta_1,\beta_2>-1\;(\beta_1\neq\beta_2),x>0,s>0$ and $\lambda\in[0,1],$ if and only if $\kappa=\infty.$
\begin{proof}
Differentiation yields
\begin{equation}\label{000}
\Gamma^2(\beta+1)\mu(x,\beta,\alpha)\frac{\partial}{\partial\beta}\psi_{\beta}^*(x,\alpha,s)=\int_0^s\frac{x^{t+\alpha}t^\beta\log(t)}{\Gamma(t+\alpha+1)}dt-\frac{\partial}{\partial\beta}\log\mu(x,\beta,\alpha)\int_0^s\frac{x^{t+\alpha}t^\beta}{\Gamma(t+\alpha+1)}dt.
\end{equation}
 We denote the difference on the right--hand side of (\ref{000}) by $\chi(s)$. Then we obtain
\begin{equation*}
\frac{\partial}{\partial s}\chi(s)=\frac{x^{s+\alpha}s^\beta}{\Gamma(s+\alpha+1)}\left[\log(s)-\frac{\partial}{\partial\beta}\log\mu(x,\beta,\alpha)\right].
\end{equation*}
This implies that there exists a number $s_5>0$ such that the function $\chi$ is decreasing on $(0,s_5)$ and increasing on $(s_5,\infty)$. Since
$$\chi(0)=\lim_{s\rightarrow\infty}\chi(s)=0,$$
we conclude that $\chi(s)\leq0$ for all $s>0.$ This implies that the  function $\beta\mapsto\psi_\beta^*(x,\alpha,s)$ is decreasing on $(-1,\infty).$ It remains to show the inequality (\ref{0000}). Since $\min(\beta_1,\beta_2)\leq \lambda\beta_1+(1-\lambda)\beta_2$ we conclude that
\begin{equation*}
\begin{split}
\psi_{\lambda\beta_1+(1-\lambda)\beta_2}^*(x,\alpha,s)&\leq\psi_{\beta}^*(x,\alpha,s)\Big(\min(\beta_1,\beta_2)\Big)\\&=\max(\psi_{\beta_1}^*(x,\alpha,s),\psi_{\beta_2}^*(x,\alpha,s))\\&=M_\infty(\psi_{\beta_1}^*(x,\alpha,s),\psi_{\beta_2}^*(x,\alpha,s)).
\end{split}
\end{equation*}
Now, we prove that the inequality (\ref{0000}) holds for all $\kappa>0.$ Then
\begin{equation}
\chi_1(s)=\lambda\left[G_{\beta_1}^*(x,\alpha,s)\right]^\kappa+(1-\lambda)\left[G_{\beta_2}^*(x,\alpha,s)\right]^\kappa-\left[G_{\lambda\beta_1+(1-\lambda)\beta_2}^*(x,\alpha,s)\right]^\kappa\geq0.
\end{equation}
Thus
\begin{equation}
\begin{split}
\frac{\Gamma(s+\alpha+1)}{\kappa x^{s+\alpha} s^{\beta_2}}\chi_1^\prime(s)&=\frac{\lambda s^{\beta_1-\beta_2}\left[G_{\beta_1}^*(x,\alpha,s)\right]^{\kappa-1}}{\mu(x,\beta_1,\alpha)}+\frac{(1-\lambda) \left[G_{\beta_2}^*(x,\alpha,s)\right]^{\kappa-1}}{\mu(x,\beta_2,\alpha)}\\&-\frac{s^{\lambda(\beta_1-\beta_2)}\left[G_{\lambda\beta_1+(1-\lambda)\beta_2}^*(x,\alpha,s)\right]^{\kappa-1}}{\mu(x,\lambda\beta_1+(1-\lambda)\beta_2,\alpha)}\\
&=\chi_2(s),\;\textrm{say.}
\end{split}
\end{equation}
If $\beta_1<\beta_2$ we have
$$\lim_{s\rightarrow\infty}\chi_2(s)=\frac{1-\lambda}{\mu(x,\beta_2,\alpha)}.$$
Hence,  there exists a number $s_6>0$ such that the function $\chi_1$ is increasing on $(s_6,\infty).$ Since
$$\lim_{s\rightarrow\infty}\chi_1(s)=0$$
we get that the function $\chi_1$ is negative on $(s_6,\infty).$ A contradiction! Therefore, the inequality (\ref{0000}) holds for all $s>0$ and $\kappa=\infty.$ This completes the proof of Theorem \ref{TTT7}.
\end{proof}
\begin{corollary}Let $\alpha,\beta>-1$ and $x>0.$ Then the function
$\beta\mapsto G_\beta^*(x,\alpha,s)$ is increasing on $(-1,\infty).$ Moreover, let  $\alpha,\beta_1,\beta_2>-1\;(\beta_1\neq\beta_2),x>0,s>0$ and $\lambda\in(0,1).$ Then the following inequality
\begin{equation}\label{00000}
G_{\lambda\beta_1+(1-\lambda)\beta_2}(x,\alpha,s)\leq M_\kappa\left(G_{\beta_1}(x,\alpha,s),G_{\beta_2}(x,\alpha,s);\lambda\right)
\end{equation}
is valid if and only if $\kappa=\infty.$
\end{corollary}
\begin{proof}Since $G_\beta+G_\beta^*=1$ and as the function $\beta\mapsto G_\beta$ is decreasing on $(-1,\infty)$ we deduce that the function $G_\beta$ is increasing on $(-1,\infty).$ Now, focus on the inequality (\ref{00000}). From $\lambda\beta_1+(1-\lambda)\beta_2\leq\max(\beta_1,\beta_2)$ we get
\begin{equation*}
\begin{split}
G_{\lambda\beta_1+(1-\lambda)\beta_2}(x,\alpha,s)&\leq G_{\beta}(x,\alpha,s)\Big(\max(\beta_1,\beta_2)\Big)\\&=\max(G_{\beta_1}(x,\alpha,s),G_{\beta_2}(x,\alpha,s))\\&=M_\infty(G_{\beta_1}(x,\alpha,s),G_{\beta_2}^*(x,\alpha,s)).
\end{split}
\end{equation*}
We suppose that the inequality (\ref{00000}) is valid for all $\kappa>0.$ Therefore,
\begin{equation}
\chi_3(s)=\lambda\left[G_{\beta_1}(x,\alpha,s)\right]^\kappa+(1-\lambda)\left[G_{\beta_2}(x,\alpha,s)\right]^\kappa-\left[G_{\lambda\beta_1+(1-\lambda)\beta_2}(x,\alpha,s)\right]^\kappa\geq0.
\end{equation}
Moreover, we have
\begin{equation}
\begin{split}
\frac{\Gamma(s+\alpha+1)}{\kappa x^{s+\alpha} s^{\beta_2}}\chi_3^\prime(s)&=\frac{s^{\lambda(\beta_1-\beta_2)}\left[G_{\lambda\beta_1+(1-\lambda)\beta_2}^*(x,\alpha,s)\right]^{\kappa-1}}{\mu(x,\lambda\beta_1+(1-\lambda)\beta_2,\alpha)}\\&-\frac{\lambda s^{\beta_1-\beta_2}\left[G_{\beta_1}^*(x,\alpha,s)\right]^{\kappa-1}}{\mu(x,\beta_1,\alpha)}-\frac{(1-\lambda) \left[G_{\beta_2}^*(x,\alpha,s)\right]^{\kappa-1}}{\mu(x,\beta_2,\alpha)}\\
&=\chi_4(s),\;\textrm{say.}
\end{split}
\end{equation}
Hence,
$$\lim_{s\rightarrow0}\chi_4(s)=\frac{\lambda-1}{\mu(x,\beta_2,\alpha)}<0.$$
Then, there exists a number $s_7$ such that the function $\chi_3$ is decreasing on $(0,s_7).$ Since
$$\lim_{s\rightarrow0}\chi_3(s)=0,$$
we deduce that the function $\chi_3$ is  negative on $(0,s_7).$ We receive a contradiction, and so $\kappa=\infty.$
\end{proof}

We needed the following Lemma to prove that the function $G_\beta$ is  subadditive.

\begin{lemma}\cite{MI}\label{l3}Let
$$h(x)=\int_x^\infty e^{-t}u(t)dt,\;f(x)=\frac{h(x)}{h(0)}.$$
If $u(x+y)/u(x)$ is increasing in $x$ on $(0,\infty)$ for every $y>0,$ then $f$ satisfies
\begin{equation*}
f(x)f(y)-f(x+y)\geq0.
\end{equation*}
If $u(x+y)/u(x)$ is decreasing in $x$ on $(0,\infty)$ for every $y>0,$ then $f$ satisfies
\begin{equation*}
f(x+y)-f(x)f(y)\geq0.
\end{equation*}
\end{lemma}

\begin{theorem}For a fixed $\alpha,\beta>-1,x>0.$  Then the function $G_{\beta}(x,\alpha,s)$
satisfies the following inequality
\begin{equation}\label{I6}
G_{\beta}(x,\alpha,s)G_{\beta}(x,\alpha,s^\prime)-G_{\beta}(x,\alpha,s+s^\prime)\geq0.
\end{equation}
\end{theorem}

\begin{proof}For a fixed $\alpha,\beta>-1, x>0,$ we define the function $u_{x,s}^{(\alpha,\beta)}(t)$ by
\begin{equation}
u_{x}^{(\alpha,\beta)}(t)=\frac{x^{t+\alpha}e^t t^\beta}{\Gamma(t+\alpha+1)\Gamma(\beta+1)}.
\end{equation}
Then the function $u_{x}^{(\alpha,\beta)}(t+\epsilon)/u_{x}^{(\alpha,\beta)}(\epsilon)$ is decreasing in $\epsilon$ on $(0,\infty)$ for every $t>0.$ Indeed, after a simple computation we have
$$\frac{u_{x}^{(\alpha,\beta)}(t+\epsilon)}{u_{x}^{(\alpha,\beta)}(\epsilon)}=\frac{x^{t+\alpha}e^t\left(1+t/\epsilon\right)^\beta\Gamma(\epsilon+\alpha+1)}{\Gamma(t+\epsilon+\alpha+1)}.$$
On the other hand, due to log--convexity property of the Gamma function $\Gamma(z),$ the ratio $z\mapsto\Gamma(z+a)/\Gamma(z)$ is increasing on $(0,\infty)$ when $a>0.$ This implies that the function $\epsilon\mapsto\frac{u_{x}^{(\alpha,\beta)}(t+\epsilon)}{u_{x}^{(\alpha,\beta)}(\epsilon)}$ is decreasing  on $(0,\infty)$ for every $\epsilon>0$ and $\beta\geq0.$ Therefore, the function $G_\beta(x,\alpha,s)$ satisfies the inequality (\ref{I6}) by means of Lemma \ref{l3}.
\end{proof}

\begin{corollary}For a fixed $\alpha>-1,\beta\geq0,x>0$
the function $G_\beta^*$ satisfies
\begin{equation}
G_\beta^*(x,\alpha,s+s\prime)-G_\beta^*(x,\alpha,s)-G_\beta^*(x,\alpha,s^\prime)+G_\beta^*(x,\alpha,s)G_\beta^*(x,\alpha,s^\prime)\geq0,
\end{equation}
for $s,t>0.$
\end{corollary}
\begin{proof}the result follows immediately by combining the inequality (\ref{I6}) with $G_\beta^*(x,\alpha,s)+G_\beta(x,\alpha,s)=1.$
\end{proof}

\noindent\textbf{Acknowledgment}

The authors like to thank Prof. Alexander Apelblat for providing us with the copy of his book on
Volterra functions, it was very useful for the preparation of this paper.

\end{document}